\documentclass[reqno]{amsart}
\usepackage{amssymb, amsmath, amsthm, amscd, mathrsfs}

\usepackage{paralist}
\usepackage{cite}
\usepackage{bigints}
\usepackage{dsfont}
\usepackage{empheq}
\usepackage{amssymb}
\usepackage{cases}
\usepackage{enumitem}
\allowdisplaybreaks

\usepackage{caption}
\usepackage{booktabs}

\usepackage{float}
\usepackage[ruled,vlined,linesnumbered,noresetcount]{algorithm2e}

\usepackage{hyperref}

\theoremstyle{plain}
\newtheorem{theorem}{Theorem}[section]
\newtheorem{corollary}{Corollary}
\newtheorem{lemma}[theorem]{Lemma}
\newtheorem{proposition}{Proposition}
\theoremstyle{definition}

\newtheorem{remark}{Remark}

\newcommand{\vertiii}[1]{{\left\vert\kern-0.25ex\left\vert\kern-0.25ex\left\vert #1 
    \right\vert\kern-0.25ex\right\vert\kern-0.25ex\right\vert}}

\def \d {\mathrm{d}}
\def \va {\varphi}
\def \la {\lambda}

\def \ppp {\partial}
\def \OOO {\Omega}

\def \N {\mathbb{N}}

\title[The backward problem for time fractional] 
      {The backward problem for time fractional evolution equations}

\author{S. E. Chorfi}
\author{L. Maniar}
\author{M. Yamamoto}

\address{S. E. Chorfi, L. Maniar, Faculty of Sciences Semlalia, LMDP, UMMISCO (IRD-UPMC), B.P. 2390, Marrakesh, Morocco}
\email{s.chorfi@uca.ac.ma, maniar@uca.ma}

\address{M. Yamamoto, Graduate School of Mathematical Sciences, The University of Tokyo, Komaba, Meguro, Tokyo 153-8914, Japan}
\address{Honorary Member of Academy of Romanian Scientists, 
Ilfov, nr. 3, Bucuresti, Romania}
\address{Correspondence member of Accademia Peloritana dei Pericolanti\\
Palazzo Universit\`a, Piazza S. Pugliatti 1 98122 Messina Italy}
\email{myama@ms.u-tokyo.ac.jp}

\makeatletter 
\@namedef{subjclassname@2020}{%
  \textup{2020} Mathematics Subject Classification}
\makeatother

 
\begin{document}
\begin{abstract}
In this paper, we consider the backward problem for fractional in time evolution equations $\partial_t^\alpha u(t)= A u(t)$ with the Caputo derivative of order $0<\alpha \le 1$, where $A$ is a self-adjoint and bounded above operator on a Hilbert space $H$. First, we extend the logarithmic convexity technique to the fractional framework by analyzing the properties of the Mittag-Leffler functions. Then we prove conditional stability estimates of H\"older type for initial conditions under a weaker norm of the final data. Finally, we give several applications to show the applicability of our abstract results.
\end{abstract}

\maketitle

\section{Introduction and main results}
Backward problems for evolution equations have many applications in science and engineering. It mainly aims at reconstructing a past profile from the present measured data of a given dynamical system. Such problems are known to be ill-posed in general, and have been extensively investigated for classical evolution equations with derivatives of integer order, see for instance \cite{AN'63, Is'17, Pa'75} and the references therein.

Recently, backward and inverse problems for time fractional evolution equations of order $0<\alpha <1$ have attracted more attention. This class of equations has proved powerful in modeling, e.g., slow and anomalous diffusion phenomena in heterogeneous or porous media \cite{AG'92, OS'74}. 

It turned out that fractional backward problems exhibit an essential difference from the classical case $\alpha=1$ in terms of their well-posedness and also the techniques used.
Needless to say, the well-posedness or the ill-posedness of the backward problems are concerned with the essence of the diffusion processes under consideration. For example, in the case where the backward problem is well-posed, it is suggested that diffusion profiles are well preserved. On the other hand, the backward problem in the case of $\alpha=1$ is severely ill-posed and the equation possesses strong smoothing property in time, and so the recovery of initial value is 
extremely difficult by the state at positive time even if time is small.

In spite of the importance, regarding the literature, general results 
on the uniqueness and stability for fractional backward problems are 
quite missing, although researches grow up rapidly. This is why we consider here an abstract framework that covers and generalizes many examples of applications.

Let $(H,\langle\cdot, \cdot\rangle)$ be a separable Hilbert space and $\|\cdot\|$ its associated norm. Let $0 < \alpha \le 1$ and $T>0$ be a fixed positive time. We are concerned with the following abstract backward problem
\begin{empheq}[left = \empheqlbrace]{alignat=2}
\begin{aligned}
&\partial_{t}^\alpha u(t) = A u(t), && \qquad  t\in (0, T), \\
& u(T)=u_T,
\end{aligned} \label{eq1}
\end{empheq}
where $A: D(A) \subset H \rightarrow H$ is a densely defined linear operator such that:\\
\textbf{Assumption I}: 
\begin{itemize}
    \item[(i)] $A$ is self-adjoint,
    \item[(ii)] $A$ is bounded above: there exists $\kappa \ge 0$ such that 
$\langle A u, u\rangle \leq \kappa \|u\|^2$ for all $u \in D(A)$,
    \item[(iii)] $A$ has compact resolvent: there exists $\lambda> -\kappa$ 
such that the resolvent $R(\lambda, A)=(\lambda I-A)^{-1}$ is compact.
\end{itemize}

The Caputo derivative $\partial_{t}^\alpha g$ is defined by
\begin{equation}\label{cap}
    \partial_t^{\alpha} g(t) = \begin{cases}\displaystyle \frac{1}{\Gamma(1-\alpha)} \int^t_0 (t-s)^{-\alpha} \frac{\d }{\d s}g(s) \, \d s, & 0<\alpha<1,\\
    \dfrac{\d }{\d t}g(t), & \alpha=1,
    \end{cases}
\end{equation}
as long as the right-hand side is defined.  Here 
$\Gamma$ is the standard Gamma function. 

First, we state a main ingredient to prove some stability results for the backward problem \eqref{eq1}.
\begin{theorem} \label{thm1}
Let $0 < \alpha \le 1$.
Let $u_T\in H$ and $u$ be the solution to \eqref{eq1}. Then there exists a constant $K\ge 1$ such that
\begin{equation}\label{lceq1}
\|u(t)\| \le K \|u(0)\|^{1-\frac{t}{T}} \|u(T)\|^{\frac{t}{T}}, \qquad 
0\le t \le T.
\end{equation}
Moreover, if $\kappa=0$, then we can choose $K=1$.
\end{theorem}

We emphasize that the stability estimate is uniform for all 
$0<\alpha \le 1$ in terms of the logarithmic convexity, 
although the behavior of $u$ for $0<\alpha<1$ near $t=0$ is expected to
be very different from the 
parabolic case $\alpha=1$.
 
We note that in the proof, the constant $K$ is given in terms of the first eigenvalue of $-A$. The estimate \eqref{lceq1} is commonly referred to as the logarithmic convexity of the solution to \eqref{eq1}. The logarithmic convexity method has been extensively investigated for backward parabolic problems, see for instance 
\cite{AN'63, ACM'21, ACM'21'', KP'60, Pa'75, YZ'01} and the references therein. The method in \cite{KP'60} relies on the complex variables to prove \eqref{lceq1}. Owing to the rapidly increasing references, we do not create a list of works and we refer to \cite{Tuan} and \cite{We} for example.
However, to the best of the authors' knowledge, Theorem \ref{thm1} is the first logarithmic convexity result for time fractional evolution equations.

Moreover, we consider
\\
\textbf{Inverse initial data problem.} 
Let $R>0$ be a fixed constant. We denote the set of admissible initial data by
\begin{align}\label{init1}
\mathcal{I}_R:=\{u_0\in D(A) \colon \|Au_0\|, \Vert u_0\Vert \le R\}. 
\end{align}
We are concerned with the stable determination of an initial datum $u_0\in \mathcal{I}_R$, where $u$ is the solution to \eqref{eq1}, from the final datum 
$u(T)$. 

More precisely, we prove conditional stability estimates of Hölder type. Conditional stability in this context means that a priori bound of the initial data is known. We emphasize that such a stability estimate is useful when dealing with numerical reconstruction of initial data, see for instance \cite{LYZ'09, YZ'01}. In practice, the original problem is replaced by an approximated one called the regularized problem. The conditional stability can be used to provide a suitable choice for the regularization parameter and to determine the rate of convergence to an exact solution of the original problem \cite{CY'00}.

Now we state two kinds of H\"older stability as Theorems \ref{thmstab} and \ref{thm3}
concerning the backward problem of determining an initial value.
\begin{theorem}\label{thmstab} 
Let $0 < \alpha \le 1$.
Let $u_0\in \mathcal{I}_R$ and let $u(t)$ be the corresponding solution to \eqref{eq1}. Then there exists a constant $\theta \in (0,\frac{1}{2})$ depending on 
$u_0$ such that the following estimate holds
\begin{equation} \label{hstab}
    \|u_0\| \le \Vert u(T)\Vert^{\theta}
\sqrt{\Vert u(T)\Vert^{2-2\theta}
+ KR^{2-2\theta}\frac{2 T^\alpha}{\alpha\Gamma(\alpha)}}. 
\end{equation}
\end{theorem}

\begin{remark}
Surprisingly, Theorem \ref{thmstab} is new even for the integer case $ \alpha=1$. Usually in this case, only a logarithmic stability is known, see e.g., \cite[p.~51]{Is'17} and \cite[Theorem 2]{ACM'21''}.
\end{remark}

In Theorem \ref{thmstab}, the exponent $\theta = \theta(u_0)$ depends on $u_0 \in \mathcal{I}_R$ and we may be able to choose
a sequence $u_0^k \in \mathcal{I}_R$, $k\in \N$ such that 
$\lim_{k\to\infty} \theta(u_0^k) = 0$, which means that 
the estimate \eqref{hstab} does not make sense in the whole 
$\mathcal{I}_R$.

We can guarantee a uniform lower bound $\theta$ if we assume 
a lower bound of $\Vert u_0\Vert$.
\begin{proposition} \label{prop1}
Let $0 < \alpha \le 1$.
We assume that $\la_n>0$ for all $n\in \N$ and that $\sum_{n=1}^{\infty}
\frac{1}{\la_n^2} < \infty$.
For arbitrarily fixed constants $R>0$ and $\delta>0$, we assume that 
a set $M \subset \mathcal{I}_R$ satisfies
$$
M \subset \{ u_0\in \mathcal{I}_R:\, \Vert u_0\Vert \ge \delta\}
$$
and that the embedding $M \subset H$ is compact.
Then there exists a constant $\theta_0(\delta) \in (0, \, \frac{1}{2})$ 
depending on $R, \delta$, such that 
$$
\|u_0\| \le K R\sqrt{1 + \frac{2T^\alpha}{\alpha\Gamma(\alpha)}} 
\|u(T)\|^{\theta_0(\delta)}
$$
for all $u_0 \in M$.
\end{proposition}

The condition $\sum_{n=1}^{\infty}\frac{1}{\la_n^2} < \infty$ is satisfied for
a regular elliptic operator $A$ in a bounded domain 
$\Omega \subset \mathbb{R}^N$, $N=1,2,3$, for example with the zero Dirichlet 
boundary condition (see Subsection \ref{sec4.1}).
In general, $\lim_{\delta \to 0}\theta_0 = 0$.

We can rewrite the proposition as
$$
\|u_0\| \le K R\sqrt{1 + \frac{2T^\alpha}{\alpha\Gamma(\alpha)}} 
\|u(T)\|^{\theta_0(\delta)} + \delta
$$
for all $u_0 \in \mathcal{I}_R$ and all $\delta>0$.
If we know the decay rate of $\theta_0(\delta)$ as 
$\delta \to 0$, then by minimizing the right-hand side of the above
inequality with respect to $\delta$, we can reach concrete 
conditional stability estimate.  However it is difficult to discuss uniformly
in $0<\alpha \le 1$, as we can understand that the case $\alpha=1$
makes solution very smooth for any small $t>0$, but $0<\alpha<1$ does not
have such strong smoothing property.

Restricted to $0<\alpha<1$, we can prove conditional 
stability of H\"older type with explicit exponent. For the statement, we introduce more general admissible sets of unknown
initial values.
Since the operator $\kappa-A$ is non-negative, the fractional powers 
$(\kappa-A)^\epsilon$ are well-defined for any $\epsilon>0$. Moreover, since $H$ is a Hilbert space, it is known that
$$D((\kappa-A)^\epsilon)=[H, D(A)]_\epsilon, \qquad \text{if } 0<\epsilon \le 1,$$
where $[\cdot, \cdot]_\epsilon$ denotes the complex interpolation bracket. We refer to \cite{MS'01} for more details. Let $\epsilon>0$ and $R>0$ be two positive constants. We consider the admissible set
\begin{align*}
\mathcal{I}_{\epsilon,R}:=\{u_0\in D((\kappa-A)^\epsilon) \colon \|u_0\|_{D((\kappa-A)^\epsilon)} \le R\}.
\end{align*}

\begin{theorem}\label{thm3}
Let $0<\alpha<1$, and 
let $\beta =\frac{\epsilon}{\epsilon + 1}$, $u_0\in \mathcal{I}_{\epsilon,R}$ and $u(t)$ being the corresponding solution to \eqref{eq1}. Then there exists a constant $C_2>0$ depending on $\alpha$ and $\lambda_{m+1}$ such that
\begin{equation}\label{eqthm3}
    \|u(0)\| \le \left(1+ C_2^{\beta} R^{1-\beta}\right) \|u(T)\|^{\beta}.
\end{equation}
\end{theorem}

Comparing to the existing literature, Sakamoto and Yamamoto \cite{SY'11} were the first to prove the well-posedness of a backward problem like \eqref{eq1} with smooth data $u_T$. More precisely, they considered the case when $H=L^2(\Omega)$, with $\Omega \subset \mathbb{R}^N$ is a smooth bounded domain, and $A$ is a symmetric uniformly elliptic operator with zero Dirichlet boundary condition. Namely, they proved that there exists a constant $C>0$ depending on $T$ such that for all $u_T \in H^2(\Omega) \cap H^1_0(\Omega)$ the corresponding solution exists and satisfies
\begin{equation}\label{lstab}
    \|u(0,\cdot)\|_{L^2(\Omega)} \le C \|u(T,\cdot)\|_{H^2(\Omega)}.
\end{equation}
This is an unconditional Lipschitz stability, that is,
a strong stability result, but under a stronger norm of the data $\|u(T,\cdot)\|_{H^2(\Omega)}$. In this particular case, although the H\"older stability 
\eqref{eqthm3} 
is a slightly weaker
estimate, \eqref{eqthm3} establishes an estimate in terms of the same norm 
$\Vert u(0,\cdot)\Vert_{L^2(\OOO)}$ and $\Vert u(T,\cdot)\Vert_{L^2(\OOO)}$.
This is crucial from a practical perspective, since only a measurement $u_T^\delta \in L^2(\Omega)$ of $u_T$ is accessible, where $\delta>0$ indicates the noise level. Note that $u_T^\delta$ is usually non-smooth due to the random noise. It should be pointed out that \eqref{lstab} has been recently extended to the non-symmetric operator case in \cite{FLY'20}.

We further note that by the interpolation inequality of the Sobolev norms,
for $\epsilon > 1$, we can derive the conclusion of Theorem \ref{thm3} 
directly from \eqref{lstab}.
 
The rest of the article is organized as follows: in Section \ref{sec2}, we start by recalling some preliminaries. Then we prove some key properties of the Mittag-Leffler function that will be used later. In Section \ref{sec3}, we prove the logarithmic convexity estimate stated in Theorem \ref{thm1}. Next, we apply the logarithmic convexity to establish Theorem \ref{thmstab} and then we prove Theorem \ref{thm3} on the Hölder stability for the inverse initial data problem. At this level, some further results are developed and the reconstruction of the solution at past times from noisy data is briefly discussed. Finally, in Section \ref{sec4}, we illustrate the abstract results by presenting several examples of time fractional systems.

\section{Preliminary results}\label{sec2}
Let us start by recalling some definitions and properties needed later.

A function $f:[0, \infty) \rightarrow (0,\infty)$ is log-convex if $\log f(t)$ is convex, i.e., for all $t, s \ge 0 \text { and } 0< \theta < 1$,
\begin{equation} \label{lc}
    f((1-\theta) t+\theta s) \leq f(t)^{1-\theta} f(s)^{\theta}.
\end{equation}
If $f(t)$ is $C^2[0,\infty)$, then it is log-convex on $[0,\infty)$ if and only if the differential inequality
\begin{equation}\label{dineq}
    f(t)f''(t)-(f'(t))^2\ge 0
\end{equation}
holds for all $t \ge 0$. It is known that the sum of two log-convex functions is also log-convex, see e.g., \cite[p.~22]{NP'18}.

A function $f:[0, \infty) \rightarrow \mathbb{R}$ is called completely monotone if $f$ has derivatives of all orders and satisfies
$$
(-1)^{k} f^{(k)}(t) \ge 0 \qquad \text{ for all } t>0, \; k=0,1,2,\ldots$$
The function $f(t)$ is completely monotone if and only if
\begin{equation}\label{cmrep}
    f(t)=\int_{0}^{\infty} \mathrm{e}^{-s t} \,\d \rho(s), \qquad t>0,
\end{equation}
where $\rho(s)$ is non-decreasing and the integral converges. We refer to \cite[Theorem 12b, p.~161]{Wi'46}. Note that the complete monotonicity is closed under addition, multiplication and pointwise convergence, see \cite[Corollary 1.6]{ScV'12}.

We will show the following key lemma (e.g., Theorem 1.3.4 (p.~21) in
\cite{NP'18}).
\begin{lemma}\label{cmlc}
Any completely monotone function $f:[0, \infty) \rightarrow (0,\infty)$ is log-convex.
\end{lemma}

\begin{proof}
By \eqref{cmrep}, we see that $\displaystyle f'(t)=-\int_{0}^{\infty} s \mathrm{e}^{-s t} \,\d \rho(s)$ and $\displaystyle f''(t)=\int_{0}^{\infty} s^2 \mathrm{e}^{-s t} \,\d \rho(s)$. Then the quadratic polynomial $\displaystyle f(t)x^2+ 2f'(t)x+f''(t)=\int_{0}^{\infty} \mathrm{e}^{-s t} (x-s)^2 \,\d \rho(s)$ is non-negative for all $x\in \mathbb{R}$. Therefore, it has a non-positive discriminant. Thus, \eqref{dineq} is satisfied. This completes the proof.
\end{proof}

Let $\alpha, \beta > 0$. We denote by $E_{\alpha,\beta}(z)$ the  Mittag-Leffler function with two parameters
$$
E_{\alpha,\beta}(z) = \sum_{k=0}^{\infty} \frac{z^k}{\Gamma(\alpha k + \beta)}, \qquad z\in \mathbb{C}.
$$
The function $E_{\alpha,\beta}(z)$ is an entire function in $\mathbb{C}$. Next, we collect some useful properties of $E_{\alpha,1}$. We refer to \cite{Pod'99} for more details. Let $\alpha> 0$ and $\lambda \ge 0$. Then
\begin{itemize}
    \item $\dfrac{\d }{\d t} E_{\alpha,1}(t)=\frac{1}{\alpha} E_{\alpha,\alpha}(t)$;
    \item $\dfrac{\d }{\d t} E_{\alpha,1}(-\lambda t^\alpha)=-\lambda t^{\alpha-1} E_{\alpha,\alpha}(-\lambda t^\alpha)$;
    \item $\partial_t^\alpha E_{\alpha,1}(-\lambda t^\alpha)=-\lambda E_{\alpha,1}(-\lambda t^\alpha)$;
    \item $\lim\limits_{t\to -\infty} E_{\alpha,1}(t)=0$.
\end{itemize}

The next lemma will be a key to prove the main result.
\begin{lemma}\label{lemlc}
Let $0 < \alpha \le \beta \le 1$ and $\lambda \ge 0$. Then the functions $t \mapsto E_{\alpha,1}(-\lambda t^{\alpha})$ and $t \mapsto t^{\beta -1} E_{\alpha,\beta}(-\lambda t^{\alpha})$ are log-convex on $[0,\infty)$.
\end{lemma}

\begin{proof}
It is known that $t \mapsto E_{\alpha,1}(-\lambda t^{\alpha})$ and $t \mapsto t^{\beta -1} E_{\alpha,\beta}(-\lambda t^{\alpha})$ are completely monotone on $[0,\infty)$, see, e.g., \cite[p.~268-269]{GM'97}. Since they are not identically zero, they are positive on $[0,\infty)$. By Lemma \ref{cmlc}, these functions are log-convex on $[0,\infty)$.
\end{proof}

\begin{remark}
Let $0 < \alpha \le 1$. The function $t \mapsto E_{\alpha,1}(-\lambda t^{\alpha})$ is not completely monotone for $\lambda<0$.
\end{remark}

\begin{remark}
Let $0<\alpha \le 1$ and $\beta \ge \alpha$. The function $t\mapsto E_{\alpha,\beta}(-t)$ is completely monotone on $[0,\infty)$, see \cite{Po'48} and \cite{Sc'96}. By the same argument as above we deduce that $t\mapsto E_{\alpha,\beta}(-t)$ is log-convex on $[0,\infty)$.
\end{remark}

\section{Logarithmic convexity and its consequences}\label{sec3}
In this section, we discuss the ill-posedness of the backward problem \eqref{eq1}. Then we prove Theorem \ref{thm1} and Theorem \ref{thmstab}. We start by recalling some useful facts. From Assumption I, it follows that the operator $A$ is closed and $(D(A), \langle\cdot, \cdot\rangle_{D(A)})$ is a Hilbert space where
$$
\langle x, y \rangle_{D(A)}=\langle x, y\rangle +\langle A x, A y\rangle, \quad \forall x, y \in D(A).
$$
Note that the operator $A$ generates a $C_0$-semigroup $(\mathrm{e}^{t A})_{t\ge 0}$ which is analytic of angle $\frac{\pi}{2}$ on $H$, see \cite[p.~106]{EN'00}. Moreover, there exists an orthonormal basis $\left\{\varphi_{n}\right\}_{n \in \mathbb{N}}$ in $H$ of eigenvectors of $-A$ such that $\varphi_{n} \in D(A),\left\|\varphi_{n}\right\|=1$ and $-A \varphi_{n}=\lambda_{n} \varphi_{n}, \forall n \in \mathbb{N}$, where $\left\{\lambda_{n}\right\}_{n \in \mathbb{N}} \subset \mathbb{R}$ denote the corresponding eigenvalues which can be ordered so that
$$
-\kappa \le \lambda_1 \le \cdots \le \lambda_m \le 0 < \lambda_{m+1} \le \cdots \, \rightarrow \infty.
$$

\subsection{Ill-posedness of the backward problem}
Let us consider the forward problem
\begin{empheq}[left = \empheqlbrace]{alignat=2}
\begin{aligned}
&\partial_{t}^\alpha u(t) = A u(t), && \qquad t\in (0, T), \\
& u(0)=u_0 .
\end{aligned} \label{eq2}
\end{empheq}
Following \cite{SY'11}, we can obtain the following well-posedness result by the Fourier method.
\begin{theorem}
Let $0<\alpha \le 1$ and $u_0 \in D(A)$. Then there exists a unique solution $u \in C\left([0, T]; D(A)\right)$ to \eqref{eq2} such that $\partial_{t}^{\alpha} u \in C\left([0, T] ; H\right)$. Moreover, the solution is given by
\begin{equation}\label{spec}
    u(t)=\sum_{n=1}^{\infty}\langle u_0, \varphi_{n}\rangle E_{\alpha, 1}\left(-\lambda_{n} t^{\alpha}\right) \varphi_{n}.
\end{equation}
\end{theorem}

From \eqref{spec}, we see that
$$u(T)=\sum_{n=1}^{\infty}\langle u_0, \varphi_{n}\rangle E_{\alpha, 1}\left(-\lambda_{n} T^{\alpha}\right) \varphi_{n}.$$
Hence, for all $n\ge 1$, we have
$$\langle u_0, \varphi_{n}\rangle E_{\alpha, 1}\left(-\lambda_{n} T^{\alpha}\right)=\langle u(T), \varphi_{n}\rangle.$$
This shows that a solution $u$ to \eqref{eq1} for $u_T\in H$ does not exist in general, and when it exists, it is unstable for non smooth data. For instance, if we choose $u_0=\frac{1}{E_{\alpha, 1}\left(-\lambda_{n} T^{\alpha}\right)} \varphi_n$, then $\|u_0\|=\frac{1}{E_{\alpha, 1}\left(-\lambda_{n} T^{\alpha}\right)}\to \infty$ as $n\to \infty$, whereas $\|u_T\|=1$. Thus, an estimate of the form $\|u_0\| \le C_1 \|u_T\|$ does not hold in general. Note that the well-posedness of the fractional backward problem ($0<\alpha <1$) holds for diffusion equations if we smooth the data $u_T$ as in \eqref{lstab}. This is considerably different from the classical parabolic case $\alpha=1$ where the problem is still unstable, even if
we adopt any Sobolev norms $\Vert u_T\Vert_{H^m(\OOO)}$ with 
arbitrary $m>0$.

\subsection{Proof of logarithmic convexity}
\begin{proof}[Proof of Theorem \ref{thm1}]
We shall use the results of Section \ref{sec2}. By the representation formula \eqref{spec}, we obtain that
\begin{align}
u(t)&=\sum_{n=1}^{m}\langle u_0, \varphi_{n}\rangle E_{\alpha, 1}\left(-\lambda_{n} t^{\alpha}\right) \varphi_{n} + \sum_{n=m+1}^{\infty}\langle u_0, \varphi_{n}\rangle E_{\alpha, 1}\left(-\lambda_{n} t^{\alpha}\right) \varphi_{n} \notag\\
& =: u_1(t) + u_2(t). \label{dceq}
\end{align}
Observe that $\|u(t)\|^{2}=\|u_1(t)\|^{2}+\|u_2(t)\|^{2}$. First, we have $\lambda_n > 0$ for $n\ge m+1$ and
$$\|u_2(t)\|^{2}=\sum_{n=m+1}^{\infty}\langle u_0, \varphi_{n}\rangle^{2} \left(E_{\alpha, 1} (-\lambda_{n} t^{\alpha})\right)^{2}.$$
In Lemma \ref{lemlc}, we have seen that the functions $t \mapsto E_{\alpha, 1} (-\lambda_{n} t^{\alpha})$ are completely on $[0,T]$ for $\lambda_n \ge 0$. By \cite[Corollary 1.6]{ScV'12}, it follows that the functions $t \mapsto \left(E_{\alpha, 1} (-\lambda_{n} t^{\alpha})\right)^{2}$ and then $t \mapsto \|u_2(t)\|^2$ are completely monotone on $[0,T]$. Thus, Lemma \ref{cmlc} implies that the function $t \mapsto \|u_2(t)\|^2$ is log-convex on $[0,T]$. Therefore
\begin{align}
    \|u_2(t)\| &\le \|u_2(0)\|^{1-\frac{t}{T}}\| u_2(T)\|^{\frac{t}{T}} \notag\\
    & \le \|u(0)\|^{1-\frac{t}{T}}\| u(T)\|^{\frac{t}{T}}. \label{lcu2}
\end{align}
On the other hand, $\displaystyle \|u_1(t)\|^{2}=\sum_{n=1}^{m}\langle u_0, 
\varphi_{n}\rangle^{2} \left(E_{\alpha, 1}
(-\lambda_{n} t^{\alpha})\right)^{2}$. 
It holds that the function $\lambda \mapsto \left(E_{\alpha, 1} (-\lambda T^{\alpha})\right)^{2}$ is non-increasing for $\lambda\le 0$. By $\lambda_1 \le \lambda_n$ for $n\le m$, we have
\begin{align*}
\Vert u_1(T)\Vert^2 &\le \max_{1\le n \le m} \vert E_{\alpha,1}(-\lambda_n
T^{\alpha})\vert^2 \sum_{n=1}^m \langle u_0, \varphi_{n}\rangle^2\\
& = \vert E_{\alpha,1}(-\lambda_1
T^{\alpha})\vert^2 \Vert u_1(0)\Vert^2\\
&=: K_1^2 \Vert u_1(0)\Vert^2.
\end{align*}
Since $\lambda_n\le 0$ for $n\le m$, the function $t \mapsto \left(E_{\alpha, 1} (-\lambda_{n} t^{\alpha})\right)^{2}$ is non-decreasing for $n\le m$. Hence, for $0\le t\le T$,
\begin{align*}
    \|u_1(t)\| &\le \|u_1(T)\| \le K_1\|u(0)\|,
\end{align*}
where the constant satisfies $K_1\ge 1$. Thus
\begin{align}
    \|u_1(t)\| &\le K_1 \|u(0)\|^{1-\frac{t}{T}}\|u_1(T)\|^{\frac{t}{T}} 
                          \notag\\
    & \le K_1 \|u(0)\|^{1-\frac{t}{T}}\| u(T)\|^{\frac{t}{T}}. \label{lcu1}
\end{align}
From \eqref{lcu2} and \eqref{lcu1}, setting $K:= K_1+1$, we obtain
$$\|u(t)\| \le K\|u(0)\|^{1-\frac{t}{T}}\|u(T)\|^{\frac{t}{T}}, \qquad 
0\le t\le T.
$$
\end{proof}

As a direct corollary of Theorem \ref{thm1}, we obtain the following backward uniqueness result.
\begin{corollary}
Let $u_T\in H$. If the solution of \eqref{eq1} satisfies $u(T)=0$, then $u(0)=0$ and hence $u(t)=0$ for all $t\in [0,T]$.
\end{corollary}
Consequently, the fractional backward system \eqref{eq1} has at most one solution.

\subsection{Proof of H\"older stability}
In order to prove Theorem \ref{thmstab} on the H\"older stability, 
we need the following lemma.
\begin{lemma}\label{dtalm}
Let $0< \alpha \le 1$. Let $u_0\in D(A)$ and $u$ be the corresponding solution to \eqref{eq1}. Then for all $t \in (0,T)$, we have
\begin{equation}\label{dtaeq}
    -\partial_t^\alpha\, \|u(t)\|^2 \le 2 \|u(t)\| \|A u_0\|.
\end{equation}
\end{lemma}

\begin{proof}
We borrow the notations from the proof of Theorem \ref{thm1}. By the formula \eqref{dceq}, we see that
\begin{equation}\label{dta1}
    -\partial_t^\alpha\, \|u(t)\|^2 = -\partial_t^\alpha\, \|u_1(t)\|^2 - \partial_t^\alpha\, \|u_2(t)\|^2.
\end{equation}
First, we have
$$
    -\partial_t^\alpha\, \|u_1(t)\|^2 = -\sum_{n=1}^{m}\langle u_0, \varphi_{n}\rangle^{2}\, \partial_t^\alpha \left[ \left(E_{\alpha, 1} (-\lambda_{n} t^{\alpha})\right)^{2}\right].
$$
Now, assume that $0<\alpha <1$. It follows from \eqref{cap} that
\begin{align}
    -\partial_t^\alpha \left[ \left(E_{\alpha, 1} (-\lambda_{n} t^{\alpha})\right)^{2}\right] &=\frac{-2}{\Gamma(1-\alpha)} \int_0^t (t-s)^{-\alpha} E_{\alpha, 1} (-\lambda_{n} s^{\alpha})\, \partial_s \left(E_{\alpha, 1} (-\lambda_{n} s^{\alpha})\right)\, \d s \notag\\
    &=\frac{2 \lambda_n}{\Gamma(1-\alpha)} \int_0^t (t-s)^{-\alpha} s^{\alpha-1} E_{\alpha, 1} (-\lambda_{n} s^{\alpha})\, E_{\alpha, \alpha} (-\lambda_{n} s^{\alpha})\, \d s . \label{mleq}
\end{align}
Since $\lambda_n \le 0$ for all $n\le m$, we note that 
$E_{\alpha,1}(-\la_ns^{\alpha}) \ge 0$ is non-decreasing in $s>0$ and 
$s^{\alpha-1}E_{\alpha,\alpha}(-\la_ns^{\alpha}) \ge 0$ for $s\ge 0$.
Hence we obtain $-\partial_t^\alpha \left[ \left(E_{\alpha, 1}
(-\lambda_{n} t^{\alpha})\right)^{2}\right]\le 0$, and 
then $-\partial_t^\alpha\, \|u_1(t)\|^2 \le 0$. On the other hand,
\begin{equation}
    -\partial_t^\alpha\, \|u_2(t)\|^2 = -\sum_{n=m+1}^{\infty}\langle u_0, \varphi_{n}\rangle^{2}\, \partial_t^\alpha \left[ \left(E_{\alpha, 1} (-\lambda_{n} t^{\alpha})\right)^{2}\right]. \label{equ2}
\end{equation}
Since $\lambda_n\ge 0$ for all $n\ge m+1$, we have
$E_{\alpha, 1} (-\lambda_{n} s^{\alpha}) \le 1$ and 
$\partial_sE_{\alpha, 1} (-\lambda_{n} s^{\alpha}) \le 0$, and 
the equality \eqref{mleq} implies that
\begin{align*}
    -\partial_t^\alpha \left[ \left(E_{\alpha, 1} (-\lambda_{n} t^{\alpha})\right)^{2}\right] &\le \frac{-2}{\Gamma(1-\alpha)} \int_0^t (t-s)^{-\alpha} \, \partial_s \left(E_{\alpha, 1} (-\lambda_{n} s^{\alpha})\right)\, \d s\\
    & = -2 \,\partial_t^\alpha \left(E_{\alpha, 1} (-\lambda_{n} t^{\alpha})\right)\\
    &= 2 \lambda_n E_{\alpha, 1} (-\lambda_{n} t^{\alpha}).
\end{align*}
Now, the above inequality, \eqref{equ2} and the 
Cauchy-Schwarz inequality yield that
\begin{align*}
    - \partial_t^\alpha\, \|u_2(t)\|^2 & \le 2 \sum_{n=m+1}^{\infty} \left(\langle u_0, \varphi_{n}\rangle  E_{\alpha, 1} (-\lambda_{n} t^{\alpha}) \right) \left(\lambda_n\langle u_0, \varphi_{n}\rangle \right)\\
    & \le 2 \|u_2(t)\| \|A u_0\|\\
    & \le 2 \|u(t)\| \|A u_0\|.
\end{align*}
The case $\alpha=1$ can be handled by simplifying the above arguments, since in this case we have
$$\|u(t)\|^2=\sum_{n=1}^{\infty}\langle u_0, \varphi_{n}\rangle^{2}\, \mathrm{e}^{-2 \lambda_n t}.$$
Thus, the proof is completed.
\end{proof}

Let $0<\alpha \le 1$. We set
$$
J_t^{\alpha}g(t):= \frac{1}{\Gamma(\alpha)}\int^t_0 (t-s)^{\alpha-1}
g(s) \, \d s, \quad g\in L^2(0,T).
$$
\begin{proof}[Proof of Theorem \ref{thmstab}]
For $g \in W^{1,1}(0,T)$, we can readily verify 
\begin{equation}\label{19}
J_t^{\alpha}\ppp_t^{\alpha}g(t) = g(t) - g(0).
\end{equation}
Next we can prove 
\begin{equation}\label{20}
\Vert u(t)\Vert^2 \in W^{1,1}(0,T) \quad \mbox{for $u_0 \in D(A)$}.
\end{equation}
Indeed, 
\begin{align*}
& \ppp_t\Vert u(t)\Vert^2
= \sum_{n=1}^{\infty} \langle u_0, \va_n\rangle^2
\ppp_t(E_{\alpha,1}(-\la_nt^{\alpha}))^2\\
=& 2 \sum_{n=1}^{\infty} \langle u_0, \va_n\rangle^2
E_{\alpha,1}(-\la_nt^{\alpha})(-\la_n)t^{\alpha-1}
E_{\alpha,\alpha}(-\la_nt^{\alpha}), \quad t>0.
\end{align*}
Dividing the summation $\sum_{n=1}^{\infty}$ into 
$\sum_{n=1}^m$ and $\sum_{n=m+1}^{\infty}$ and using $\la_n > 0$ for
$n\ge m+1$, $\la_n \le 0$ for $1\le n \le m$, we can apply 
Theorem 1.6 (p.~35) in \cite{Pod'99} for $n\ge m+1$, and
$\vert E_{\alpha,1}(-\la_nt^{\alpha})\vert$,
$\vert E_{\alpha,\alpha}(-\la_nt^{\alpha})\vert \le C_2$ for 
$0\le t\le T$ and $1\le n \le m$, we obtain
\begin{align*}
& \vert \ppp_t\Vert u(t)\Vert^2 \vert 
\le C_3 \sum_{n=1}^m \vert \la_n\vert \langle u_0,\va_n\rangle^2t^{\alpha-1}
+ C_3 \sum_{n=m+1}^{\infty} \vert \la_n\vert \langle u_0,\va_n\rangle^2
t^{\alpha-1}\left( \frac{1}{1+\la_nt^{\alpha}}\right)^2\\
\le& C_4\left( \sum_{n=1}^{\infty} \vert \la_n\vert \langle u_0,\va_n\rangle^2
\right)t^{\alpha-1} \in L^1(0,T)
\end{align*}
by $\alpha>0$.
Thus the verification of \eqref{20} is complete.
\\

By the Sobolev embedding, we remark that $\Vert u(t)\Vert \in 
C[0,T]$. Since $J_t^{\alpha}g_1 \le J_t^{\alpha}g_2$ in $(0,T)$ for 
$g_1 \le g_2$ in $(0,T)$, we operate $J_t^{\alpha}$ to both sides of
\eqref{dtaeq}. In terms of \eqref{19}, we see
$$
\Vert u(0)\Vert^2 - \Vert u(T)\Vert^2
\le \frac{2}{\Gamma(\alpha)}\int^T_0 (T-s)^{\alpha-1}
\Vert u(s)\Vert \, \d s \, \Vert Au_0\Vert.
$$
Since $(T-s)^{\alpha-1} \in L^1(0,T)$ as function in $s$ and 
$\Vert u(s)\Vert \in C[0,T]$ by \eqref{20}, the 
mean value theorem for integral yields that there exists $\xi=\xi(u_0) 
\in (0,T)$ such that 
$$
\int^T_0 (T-s)^{\alpha-1} \Vert u(s)\Vert \,\d s
= \Vert u(\xi)\Vert\int^T_0 (T-s)^{\alpha-1} \,\d s
= \frac{T^{\alpha}}{\alpha}\Vert u(\xi)\Vert.
$$
Therefore, we obtain
$$
\|u(0)\|^{2} \le \|u(T)\|^{2} + \frac{2\, T^\alpha}{\alpha\Gamma(\alpha)}
\|u(\xi)\| \Vert Au_0\Vert.
$$
By Theorem \ref{thm1}, we have
$$
\Vert u(\xi)\Vert \le K\Vert u(0)\Vert^{1-\frac{\xi}{T}}
\Vert u(T)\Vert^{\frac{\xi}{T}} \le KR^{1-\frac{\xi}{T}}
\Vert u(T)\Vert^{\frac{\xi}{T}}.
$$
Therefore, setting $\theta := \frac{\xi}{2T}$, we reach the conclusion.
This achieves the proof.
\end{proof}

\begin{proof}[Proof of Proposition \ref{prop1}]
By $u(u_0) = u(u_0)(t)$ we denote the solution to 
$$
\partial_t^{\alpha} u(t) = Au(t), \quad u(0)=u_0.
$$
Then we know that for $u_0\in H$, there exists a unique solution 
$u(u_0) \in C([0,T];H)$ (e.g., Sakamoto and Yamamoto \cite{SY'11} for 
$0<\alpha<1$ and Pazy \cite{Pa} for $\alpha=1$).
Also, it is known that 
$$
\Vert u(u_0)(t) - u_0\Vert^2 =
\begin{cases}
\displaystyle \sum_{n=1}^{\infty} \langle u_0, \va_n \rangle^2 
(E_{\alpha,1}(-\la_nt^{\alpha}) -1)^2, & 0<\alpha<1,\\
\displaystyle \sum_{n=1}^{\infty} \langle u_0, \va_n \rangle^2 
(\mathrm{e}^{-\la_nt} -1)^2, & \alpha=1.
\end{cases}
$$
Assume that the conclusion of the theorem does not hold.
By $\theta(u_0)$, we denote the exponent $\theta$ in (5). 
Then there exist $u_0^k \in M$ such that $\lim_{k\to\infty} 
\theta(u_0^k) = 0$.
Set $\xi_k:= 2T\theta(u_0^k)$. 
Moreover, the compactness implies that a subsequence, 
still denoted by $u_0^k$, satisfies $\lim_{k\to \infty}u_0^k 
= u_0 \in H$ with some $u_0$.
In 
\begin{equation}\label{*}
\int^T_0 (T-s)^{\alpha-1}\Vert u_k(s)\Vert\, \d s = 
\Vert u_k(\xi_k)\Vert \frac{T^{\alpha}}{\alpha},
\end{equation}
we set $u_k(t):= u(u_0^k)(t)$ for $k\in \N$ and $t>0$.
We write $u=u(u_0)$.
We discuss only the case $0<\alpha<1$, and the case $\alpha=1$ can be 
discussed in the same way.
We have
\begin{align*}
& \Vert u_k(\xi_k) - u(0)\Vert
= \Vert u_k(\xi_k) - u_k(0) + u_k(0) - u(0)\Vert \\
\le &\Vert u_k(\xi_k) - u_k(0)\Vert + \Vert u_k(0) - u_0\Vert,
\end{align*}
and 
$$
\Vert u_k(\xi_k) - u_k(0)\Vert^2 
= \sum_{n=1}^{\infty} \langle u_0^k, \va_n \rangle^2 
(E_{\alpha,1}(-\la_n\xi_k^{\alpha}) -1)^2
$$
and
$$
\lim_{k\to \infty} \Vert u_k(0) -  u_0\Vert 
= \lim_{k\to \infty} \Vert u_0^k - u_0\Vert = 0.
$$
Since $\Vert Au_0^k\Vert \le R$ for all $k\in \N$, we have
\begin{align*}
& \vert \langle u_0^k, \va_n \rangle \vert
= \vert \langle A^{-1}Au_0^k, \va_n\rangle\vert 
= \vert \langle Au_0^k, A^{-1}\va_n\rangle\vert
= \left\vert \frac{1}{\la_n}\langle Au_0^k, \va_n\rangle \right\vert\\
\le& \frac{1}{\la_n}\Vert Au_0^k\Vert \Vert \va_n\Vert 
\le \frac{R}{\la_n}.
\end{align*}
Also using the assumption $\sum_{n=1}^\infty\frac{1}{\la_n^2} < \infty$, we 
can 
apply the Lebesgue convergence theorem to see that 
$\lim_{k\to\infty} \Vert u_k(\xi_k) - u_k(0)\Vert = 0$.

Therefore, letting $k \to \infty$ in \eqref{*},
we obtain
$$
\int^T_0 (T-s)^{\alpha-1}(\Vert u(s)\Vert-\Vert u(0)\Vert)\, \d s = 0.
$$
By $\la_n> 0$, using 
$$
\Vert u(s)\Vert^2 = \sum_{n=1}^{\infty} \vert \langle u_0, \va_n
\rangle\vert^2 E_{\alpha,1}(-\la_ns^{\alpha})^2
\le \sum_{n=1}^{\infty} \vert \langle u_0, \va_n \rangle\vert^2,
$$
we see that $\Vert u(s)\Vert \le \Vert u(0)\Vert$ for 
$s>0$.  Hence $\Vert u(s)\Vert = \Vert u(0)\Vert$ for all $s>0$.
Therefore the eigenvectors expansion yields
$$
\sum_{n=1}^{\infty} \vert \langle u_0, \va_n\rangle \vert^2
(1 - E_{\alpha,1}(-\la_ns^{\alpha}))^2 = 0.
$$
Since $\la_n > 0$, we have $E_{\alpha,1}(-\la_ns^{\alpha}) < 1$ for all
$n\in \N$. Therefore, $\langle u_0, \va_n\rangle = 0$ for all 
$n\in \N$, which means that $u_0 = 0$.
Since $\lim_{k\to\infty} \Vert u_0^k\Vert= \Vert u_0\Vert$ and 
$\Vert u_0^k\Vert
\ge \delta$ for $k\in \N$ by $u_0^k \in M$, we see that $0 \ge \delta$. Hence we reach a contradiction.  Thus the proof of Proposition \ref{prop1} is complete.
\end{proof}

\begin{proof}[Proof of Theorem \ref{thm3}]
We set
\begin{align*}
u(t)&=\sum_{n=1}^{m}\langle u_0, \varphi_{n}\rangle E_{\alpha, 1}\left(-\lambda_{n} t^{\alpha}\right) \varphi_{n} + \sum_{n=m+1}^{\infty}\langle u_0, \varphi_{n}\rangle E_{\alpha, 1}\left(-\lambda_{n} t^{\alpha}\right) \varphi_{n} \notag\\
& =: u_1(t) + u_2(t).
\end{align*}
Then
\begin{align*}
    \|u(0)\|^{2}&=\sum_{n=1}^{m}\langle u_0, \varphi_{n}\rangle^{2} + \sum_{n=m+1}^{\infty}\langle u_0, \varphi_{n}\rangle^{2} \\
    & =: \|u_1(0)\|^{2} + \|u_2(0)\|^{2}.
\end{align*}
First, since $\lambda_n\le 0$ for $n\le m$, the function $t \mapsto E_{\alpha, 1} (-\lambda_{n} t^{\alpha})$ is non-decreasing for $n\le m$. Then $E_{\alpha, 1} (-\lambda_{n} t^{\alpha}) \ge 1$ for $n\le m$ and $\|u_1(0)\|^{2} \le \|u_1(T)\|^{2} \le \|u(T)\|^{2}$. On the other hand, by using the 
H\"older inequality we obtain
\begin{align*}
    \|u_2(0)\|^{2}&= \sum_{n=m+1}^{\infty}\langle u_0, \varphi_{n}\rangle^{2} \\
    & = \sum_{n=m+1}^{\infty} \frac{\langle u(T), \varphi_{n}\rangle^{2}}{\left(E_{\alpha, 1}\left(-\lambda_{n} T^{\alpha}\right)\right)^2}\\
    & = \sum_{n=m+1}^{\infty} \frac{\langle u(T), \varphi_{n}\rangle^{2(1-\beta)}}{\left(E_{\alpha, 1}\left(-\lambda_{n} T^{\alpha}\right)\right)^{2}} \; \langle u(T), \varphi_{n}\rangle^{2\beta}\\
    & \le \left(\sum_{n=m+1}^{\infty} \frac{\langle u(T), \varphi_{n}\rangle^{2}}{\left(E_{\alpha, 1}\left(-\lambda_{n} T^{\alpha}\right)\right)^{\frac{2}{1-\beta}}}\right)^{1-\beta} \left(\sum_{n=m+1}^{\infty} \langle u(T), \varphi_{n}\rangle^{2}\right)^\beta\\
    & \le \left(\sum_{n=m+1}^{\infty} \frac{\langle u_0, \varphi_{n}\rangle^{2}}{\left(E_{\alpha, 1}\left(-\lambda_{n} T^{\alpha}\right)\right)^{\frac{2\beta}{1-\beta}}}\right)^{1-\beta} \|u(T)\|^{2\beta}.
\end{align*}
Since $-\lambda_{n} T^{\alpha} \le 0$ for $n\ge m+1$ and $0<\alpha<1$, 
by Theorem 1.4 (p.~33-34) in \cite{Pod'99}, we see
$$
E_{\alpha,1}(-\la_nT^{\alpha}) = \frac{1}{\Gamma(1-\alpha)\la_nT^{\alpha}}
+ O\left( \frac{1}{\la_n^2}\right)
$$
for all large $n \in \N$.  Consequently we can prove
$$
\left(E_{\alpha, 1}\left(-\lambda_{n} T^{\alpha}\right)\right)^{-1} \le C_1 (1+\lambda_n T^\alpha)
$$
for some constant $C_1>0$.  See also \cite[Corollary 2.2]{LY'10}.

Setting $C_2=C_1\left(\frac{1}{\lambda_{m+1}}+T^\alpha\right)$, we obtain
$$
\left(E_{\alpha, 1}\left(-\lambda_{n} T^{\alpha}\right)\right)^{-1} 
\le C_2\lambda_n
$$
for $n \ge m+1$. Therefore
\begin{align*}
    \|u_2(0)\|^{2}& \le C_2^{2\beta} \left(\sum_{n=m+1}^{\infty} \langle u_0, \varphi_{n}\rangle^{2} \lambda_n^{\frac{2\beta}{1-\beta}}\right)^{1-\beta} \|u(T)\|^{2\beta}\\
    & \le C_2^{2\beta} \|u_0\|_{D((\kappa-A)^\epsilon)}^{2(1-\beta)} \|u(T)\|^{2\beta},
\end{align*}
since $\epsilon=\frac{\beta}{1-\beta}$. We assume without loss of generality that $\|u(T)\| \le 1$. Hence
\begin{align*}
    \|u(0)\| & \le \|u_1(0)\| + \|u_2(0)\|\\
    & \le \|u(T)\| + C_2^{\beta} R^{1-\beta} \|u(T)\|^{\beta}\\
    & \le \left(1+ C_2^{\beta} R^{1-\beta}\right) \|u(T)\|^{\beta}.
\end{align*}
\end{proof}

\begin{remark}
Setting $\epsilon=1$, i.e., $\beta=\frac{1}{2}$ in Theorem \ref{thm3}, we obtain a similar result to Theorem \ref{thmstab} for $0<\alpha<1$, with $\theta=\frac{1}{2}$ and a different constant. 
Note that Theorem \ref{thm3} does not cover $\alpha=1$, since we have used the 
inequality: for $0<\alpha<1$, there exists a constant $c>0$ such that 
$(1-x) E_{\alpha,1}(x) \ge c>0$ for all $x\le 0$, which does not hold 
for $\alpha=1$.
\end{remark}

\subsection{Backward problems from noisy data}
The purpose of this part is to reconstruct the solution $u(t)$ for $0<t<T$ to the following problem
\begin{empheq}[left = \empheqlbrace]{alignat=2}
\begin{aligned}
&\partial_{t}^\alpha u(t) = A u(t) + f(t), && \qquad t\in (0, T), \\
& u(T)=u_T,
\end{aligned} \label{eq3}
\end{empheq}
from a noisy measurement $u_T^\delta$ of the exact final datum $u(T)$, i.e.,
\begin{equation*}
    \|u(T)-u_T^\delta\| \le \delta,
\end{equation*}
where $\delta>0$ is some known measurement error. Here the source term $f$ is assumed to be known. Then an approximate solution to the original problem is given by
\begin{empheq}[left = \empheqlbrace]{alignat=2}
\begin{aligned}
&\partial_{t}^\alpha u^\delta(t) = A u^\delta(t) + f(t), && t\in (0, T), \\
& u^\delta(T)=u_T^\delta.
\end{aligned} \label{eq4}
\end{empheq}
Let $R>0$ be a known a priori bound. Let us set the following class of solutions
\begin{equation}
    \mathcal{M}=\left\{u \in C([0,T]; H) \colon \|u(0)\| \le R \right\}.
\end{equation}
Restricted to the class $\mathcal{M}$, the logarithmic convexity estimate \eqref{lceq1} yields a H\"older continuous dependence of the solution $u(t)$, $0<t<T$, with respect to the data. More precisely, we have the following result.
\begin{theorem}
Let $u$ and $u^\delta$ be the solutions 
to \eqref{eq3} and \eqref{eq4} respectively such that $u, u^{\delta}
\in \mathcal{M}$. Then
\begin{equation}
    \|u(t)-u^\delta(t)\| \le 2 K R^{1-\frac{t}{T}} \delta^{\frac{t}{T}}, \qquad 0\le t \le T.
\end{equation}
\end{theorem}

\begin{proof}
If $v(t)=u(t)-u^\delta(t)$, then $v$ is a solution to \eqref{eq1} corresponding to the final datum $v(T)=u(T)-u_T^\delta$ and the initial datum $v(0)=u_0-u_0^\delta$. Since $\|v(0)\| \le 2 R$, the estimate \eqref{lceq1} applied to $v$ yields the desired result.
\end{proof}

As a widely used numerical method for reconstructing 
an initial value $u_0$ from noisy data at $t=T$, we can refer to 
the Tikhonov regularization, where one important issue is choices of 
the regularizing parameters according to given noise level 
$\delta>0$.  The conditional stability Theorem \ref{thm3} provides 
a choice principle for quasi-optimal convergence of the scheme, and 
as for the details, see Cheng and Yamamoto \cite{CY'00}. 

\section{Applications} \label{sec4}
In this section, we present some various examples of time fractional systems that fit in our abstract framework and for which we can immediately apply 
Theorems \ref{thm1}, \ref{thmstab} and \ref{thm3}.

\subsection{Fractional uniformly elliptic equations}\label{sec4.1}
Henceforth, $\Omega\subset \mathbb{R}^N$ is a bounded domain with boundary $\partial \Omega$ of class $C^2$. We consider the backward problem with Dirichlet or Neumann boundary conditions:
\begin{empheq}[left = \empheqlbrace]{alignat=2}
\begin{aligned}
&\partial_{t}^\alpha u(t,x) = A u(t,x), && \qquad\text { in }  (0, T) \times \Omega, \\
&\begin{cases}u\rvert_{\partial \Omega} =0, &\quad\text { on } (0, T) \times \partial \Omega \;\text{ (Dirichlet case) or}\\
\partial_\nu^A u =0, &\quad\text { on } (0, T) \times \partial \Omega \;\,\text{ (Neumann case)},
    \end{cases}\\
& u(T,x)=u_T(x) && \qquad \text{ in } \Omega,
\end{aligned} \label{eqe1}
\end{empheq}
where the state space is $H=L^2(\Omega)$, and $A$ is a symmetric uniformly elliptic operator, i.e.,
\begin{equation}\label{opdef}
Au\, (x) := \sum_{i,j=1}^N \partial_i(a_{ij}(x)\partial_j u)(x) + p(x) u,
\end{equation}
such that
$$
a_{ij} = a_{ji} \in C^1(\overline{\Omega}),
\quad 1\le i,j \le N, \qquad  p \in L^\infty(\Omega),
$$
and there exists a constant $\kappa>0$ such that 
$$
\sum_{i,j=1}^N a_{ij}(x)\xi_i\xi_j \ge \kappa \sum_{j=1}^N \xi_j^2,
\quad x \in \overline{\Omega}, \; (\xi_1, ..., \xi_N) \in \mathbb{R}^N.
$$
The unit outwards normal vector to $\partial \Omega$ is denoted by $\nu$ and the conormal derivative with respect to $A$ by
$$\partial_{\nu}^{A} u=\sum_{i, j=1}^{N} a_{i j} \nu_{j} (\partial_{i} u)_{|\partial \Omega}.$$
The domain of the operator $A$ is given by
\begin{equation} \label{domop}
    D(A)= \begin{cases}
H^2(\Omega) \cap H^1_0(\Omega), & (\text{Dirichlet case) or}\\
\left\{u\in H^2(\Omega) \colon \partial_{\nu}^{A} u=0 \right\} & 
(\text{Neumann case}).
\end{cases}
\end{equation}
It is known that the operator $A$ satisfies Assumption I. Indeed, it is self-adjoint and bounded above thanks to Green's formula. Since the embedding $D(A) \hookrightarrow L^2(\Omega)$ is compact, $A$ has compact resolvent, see \cite[Proposition 4.25]{EN'00}.
Then Theorem \ref{thmstab} yields
\begin{corollary}\label{cor2}
Let $u_0\in \mathcal{I}_R$ and let $u(t, x)$ be the corresponding solution to \eqref{eqe1}. Then there exists a constant $\theta \in (0,\frac{1}{2})$ such that
\begin{equation*}
    \|u(0,\cdot)\|_{L^2(\Omega)} \le K R^{1-\theta}\sqrt{1 + \frac{2 \,  T^\alpha}{\alpha\Gamma(\alpha)}} \|u(T,\cdot)\|_{L^2(\Omega)}^\theta.
\end{equation*}
\end{corollary}
We can apply Theorem \ref{thm3}, but we omit the statements of stability estimates
which Theorem \ref{thm3} brings in this and the succeeding cases.

\subsection{Fractional advection-diffusion equation}
Let $l>0$, $b \in \mathbb{R}$ and $d>0$. We consider the following one-dimensional problem
\begin{equation}\label{adeq}
    \begin{cases}
    \partial^{\alpha}_t u(t,x)=d u_{xx}(t,x) -b u_{x}(t,x), & (t,x)\in (0,T) 
    \times (0,l),\\
    u(t,0)=u(t,l)=0, & t\in (0,T),\\
    u(T, x)=u_T(x), & x\in (0,l).
    \end{cases}
\end{equation}
Although the equation \eqref{adeq} is not symmetric, it can be reduced to the form \eqref{eqe1}. Let us make the change of variables $\xi=\frac{x}{\sqrt{d}}$ and $u(t,x)=v(t,\xi) \,\mathrm{e}^{\frac{b\xi}{2\sqrt{d}}}$. By a simple calculation, \eqref{adeq} can be transformed to the following equation
\begin{equation}\label{adeq1}
    \begin{cases}
    \partial^{\alpha}_t v(t,\xi)=v_{\xi \xi}(t,\xi) + p v(t,\xi), & (t,\xi)\in (0,T) \times (0,\ell),\\
    v(t,0)=v(t,\ell)=0, & t\in (0,T),\\
    v(T, \xi)=u_{T}(\xi\sqrt{d}) \,\mathrm{e}^{\frac{-b\xi}{2\sqrt{d}}}, & x\in (0,\ell),
    \end{cases}
\end{equation}
where $p=-\frac{b^2}{4d}$ and $\ell=\frac{l}{\sqrt{d}}$. 
Here we note $v(0,\xi) = u(0, \xi\sqrt{d})\,\mathrm{e}^{\frac{-b\xi}{2\sqrt{d}}}$.
From Corollary \ref{cor2} applied to \eqref{adeq1}, we obtain
\begin{corollary}
Let $u_0\in \mathcal{I}_R$ and let $u(t, x)$ be the corresponding solution to \eqref{adeq}. Then there exist $\theta \in (0,\frac{1}{2})$ and a positive constant $C=C(b,l,d,\theta)$ such that
\begin{equation*}
    \|u(0,\cdot)\|_{L^2(0,l)} \le CK R^{1-\theta}\sqrt{1 + \frac{2 \,  T^\alpha}{\alpha\Gamma(\alpha)}} \|u(T,\cdot)\|_{L^2(0,l)}^\theta.
\end{equation*}
\end{corollary}

\subsection{Space-time fractional diffusion equations}
Let $0<s<1$. We consider the backward problem
\begin{empheq}[left = \empheqlbrace]{alignat=2}
\begin{aligned}
&\partial_{t}^\alpha u(t,x) = -(-A_0)^s u(t,x), && \qquad\text { in }  (0, T) \times \Omega, \\
& u\rvert_{\partial \Omega} =0, && \qquad\text { on } (0, T) \times \partial \Omega, \\
& u(T,x)=u_T(x) && \qquad \text{ in } \Omega,
\end{aligned} \label{steq1}
\end{empheq}
where (as in \eqref{opdef}) $A_0$ is a symmetric uniformly elliptic operator given by
\begin{equation*}
A_0 u\, (x) := \sum_{i,j=1}^N \partial_i(a_{ij}(x)\partial_j u)(x), \qquad D(A_0)= H^2(\Omega) \cap H^1_0(\Omega).
\end{equation*}
The spectrum $\sigma(-A_{0})$ consists of eigenvalues with finite multiplicities so that we can write
$$
0<\lambda_{1} \le \lambda_{2} \leq \lambda_{3}<\cdots \rightarrow \infty .
$$
Let $\left\{\varphi_{n}\right\}_{n \in \mathbb{N}}$ be an orthonormal basis of associated eigenfunctions in $L^2(\Omega)$. Then the spectral fractional power $(-A_0)^s$ is defined by
$$(-A_0)^s u=\sum_{n=1}^\infty \lambda_n^s \langle u, \varphi_n \rangle_{L^2(\Omega)} \varphi_n,$$
with domain
$$D((-A_0)^s)=\left\{u\in L^2(\Omega) \colon \sum_{n=1}^\infty \lambda_n^{2s} \langle u, \varphi_n\rangle_{L^2(\Omega)}^2 < \infty\right\},$$
endowed with the norm $\|u\|_{D((-A_0)^s)}=\|(-A_0)^s u\|_{L^2(\Omega)}$ for all $u \in D((-A_0)^s)$. From the definition, it can be proved that the operator $-(-A_0)^s$ satisfies Assumption I. Hence, we have the following estimate.
\begin{corollary}
Let $u_0\in \mathcal{I}_R$ and let $u(t, x)$ be the corresponding solution to \eqref{steq1}. Then there exists a constant $\theta \in (0,\frac{1}{2})$ such that
\begin{equation*}
    \|u(0,\cdot)\|_{L^2(\Omega)} \le K R^{1-\theta}\sqrt{1 + \frac{2 \,  T^\alpha}{\alpha\Gamma(\alpha)}} \|u(T,\cdot)\|_{L^2(\Omega)}^\theta.
\end{equation*}
\end{corollary}

\subsection{Fractional coupled diffusion systems}
We consider the coupled system of diffusion equations with Dirichlet or Neumann boundary conditions:
\begin{empheq}[left = \empheqlbrace]{alignat=2}
{\small\begin{aligned}
&\partial_{t}^\alpha u_1(t,x) = A_0^1 u_1(t,x)+\sum_{j=1}^N c_{1j}(x) u_j (t,x), && \text{in }  (0, T) \times \Omega, \\
&\partial_{t}^\alpha u_2(t,x) = A_0^2 u_2(t,x)+\sum_{j=1}^N c_{2j}(x) u_j (t,x), && \text{in }  (0, T) \times \Omega, \\
&\cdots\\
&\partial_{t}^\alpha u_N(t,x) = A_0^N u_N(t,x)+\sum_{j=1}^N c_{Nj}(x) u_j (t,x), && \text{in }  (0, T) \times \Omega, \\
&\begin{cases}u_i\rvert_{\partial \Omega} =0, &\text {on } (0, T) \times \partial \Omega \;\text{ (Dirichlet case)}, 
\hspace{0.1cm} 1\le i \le N, \text{ or}\\
\partial_\nu^{A_0^i} u_i =0, &\text {on } (0, T) \times \partial \Omega \;\,\text{ (Neumann case)}, \;\, 1\le i\le N,
    \end{cases}\\
& u_i(T,x)=u_T^i(x), \hspace{1.6cm} 1\le i \le N,  && \text{in } \Omega,
\end{aligned}} \label{seqe1}
\end{empheq}
where, for every $1\le k\le N$, $A_0^k$ is a symmetric uniformly elliptic operator,
\begin{equation*}
A_0^k u := \sum_{i,j=1}^N \partial_i(a_{ij}^k(x)\partial_j u),
\end{equation*}
such that
$$
a_{ij}^k = a_{ji}^k \in C^1(\overline{\Omega}),
\quad 1\le i,j \le N, \qquad  c_{ij} \in L^\infty(\Omega), \; c_{ij}(\cdot)=c_{ji}(\cdot),
$$
and there exists a constant constant $\kappa_k>0$ such that 
$$
\sum_{i,j=1}^N a_{ij}^k(x)\xi_i\xi_j \ge \kappa_k \sum_{j=1}^N \xi_j^2,
\quad x \in \overline{\Omega}, \; (\xi_1, ..., \xi_N) \in \mathbb{R}^N.
$$
Moreover, the domain of each operator $A_0^k$ is given by \eqref{domop}. The system \eqref{seqe1} can be written as
\begin{empheq}[left = \empheqlbrace]{alignat=2}
\begin{aligned}
&\partial_{t}^\alpha \mathbf{u}(t) = \mathcal{A}_0 \mathbf{u}(t) + \mathbf{C} \mathbf{u}(t), && \qquad  t\in (0, T), \\
& \mathbf{u}(T)=\mathbf{u}_T, \nonumber
\end{aligned}
\end{empheq}
where the state space is $H=L^2(\Omega;\mathbb{R}^N)$, 
$\mathbf{u}_T=(u_T^i)_{1\le i\le N}$, $\mathbf{u}=(u_i)_{1\le i\le N}$, $\mathbf{C}(\cdot)=(c_{ij}(\cdot))_{1\le i,j\le N} \in L^\infty(\Omega; \mathcal{L}(\mathbb{R}^N))$ and $\mathcal{A}_0=\mathrm{diag}(A_0^1, \cdots, A_0^N)$.

We can directly prove that the operator $\mathcal{A}=\mathcal{A}_0 +\mathbf{C}$ satisfies Assumption I. Consequently, we have
\begin{corollary}\label{scor2}
Let $\mathbf{u}_0\in \mathcal{I}_R$ and let $\mathbf{u}(t)$ be the corresponding solution to \eqref{seqe1}. Then there exists a constant $\theta \in (0,\frac{1}{2})$ such that
\begin{equation*}
    \|\mathbf{u}(0)\|_{L^2(\Omega;\mathbb{R}^N)} \le K R^{1-\theta}\sqrt{1 + \frac{2 \,  T^\alpha}{\alpha\Gamma(\alpha)}} \|\mathbf{u}(T)\|_{L^2(\Omega;\mathbb{R}^N)}^\theta.
\end{equation*}
\end{corollary}

\subsection{Fractional degenerate equations}
Let $0 \leq \beta < 2$. We consider the fractional degenerate problem
\begin{equation} \label{deq}
\begin{cases} \partial_t^\alpha u=\left(x^{\beta} u_{x}\right)_{x}, & (t, x) \in (0,T) \times (0,1),\\
u(t, 1)=0, \qquad \begin{cases} u(t, 0)=0, &  \text{ if } \beta\in [0, 1),\\
\left(x^{\beta} u_{x}\right)(t, 0)=0, & \text{ if } \beta\in [1, 2),
\end{cases}\\
u(T, x)=u_{T}(x) & x \in(0, 1).
\end{cases}
\end{equation}
The state space is $H=L^2(0,1)$ and the governing operator is given by
$$
A u =\left(x^{\beta} u_{x}\right)_{x},
$$
with domain:\\
If $\beta\in [0, 1)$ (weakly degenerate), then
$$D(A)=\left\{u \in H_{\beta, 0}^{1}(0,1): x^{\beta} u_x \in H^{1}(0,1)\right\},$$
where
$$
H_{\beta, 0}^{1}(0,1)=\left\{u \in L^{2}(0,1): u \in AC([0,1]), x^{\frac{\beta}{2}} u_x \in L^{2}(0,1), u(0)=0, u(1)=0\right\}.
$$
If $\beta\in [1, 2)$ (strongly degenerate), then
$$D(A)=\left\{u \in H_{\beta, 0}^{1}(0,1): x^{\beta} u_x \in H^{1}(0,1)\right\},$$
where we changed the definition to
$$
H_{\beta, 0}^{1}(0,1)=\left\{u \in L^{2}(0,1): u \in AC((0,1]), x^{\frac{\beta}{2}} u_x \in L^{2}(0,1), u(1)=0 \text{ and } x^{\beta} u_x\in H^1(0,1)\right\}.
$$
$AC(I)$ denotes the space of absolutely continuous functions on the interval $I$.

In both cases, the operator $A$ satisfies Assumption I, see, e.g., \cite{CMP'98, CMV'05} for (i)-(ii) and \cite[Appendix]{ACF'06} for (iii). Consequently, we infer the following corollary.
\begin{corollary}
Let $u_0\in \mathcal{I}_R$ and let $u(t, x)$ be the corresponding solution to \eqref{deq}. Then there exists a constant $\theta \in (0,\frac{1}{2})$ such that
\begin{equation*}
    \|u(0,\cdot)\|_{L^2(0,1)} \le K R^{1-\theta}\sqrt{1 + \frac{2 \,  T^\alpha}{\alpha\Gamma(\alpha)}} \|u(T,\cdot)\|_{L^2(0,1)}^\theta.
\end{equation*}
\end{corollary}

\subsection{Fractional dynamic boundary conditions}
Let us set $\Gamma:=\partial \Omega$. We consider the backward parabolic system with dynamic boundary conditions
\begin{empheq}[left = \empheqlbrace]{alignat=2}\label{dy1}
\begin{aligned}
&\partial_{t}^\alpha u(t,x)=\Delta u(t,x), && \qquad\text { in } (0, T)\times \Omega, \\
&\partial_{t}^\alpha u_{\Gamma}(t,x) = \Delta_{\Gamma} u_{\Gamma}(t,x) - \partial_{\nu}u(t,x), && \qquad\text { on } (0, T)\times \Gamma, \\
& u_{\Gamma}(t,x) = u_{|\Gamma}(t,x), &&\qquad\text{ on } (0, T)\times \Gamma, \\
& \left(u(T, \cdot),u_{\Gamma}(T, \cdot)\right)=\left(u_T,u_{T,\Gamma}\right), && \qquad \text{ on } \Omega\times\Gamma,
\end{aligned}
\end{empheq}
where $\Delta_\Gamma$ is the the Laplace-Beltrami operator on $\Gamma$. The system \eqref{dy1} can be written in the state space
$\mathbb{L}^2:=L^2(\Omega) \times L^2(\Gamma)$ as:
$$
\left\{\begin{array}{l}
\hspace{-0.2cm}\partial_{t}^\alpha \mathbf{U}(t)=\mathbf{A} \mathbf{U}(t), \qquad t \in (0,T), \\
\hspace{-0.2cm} \mathbf{U}(T)=\mathbf{U}_T,
\end{array}\right.
$$
where $\mathbf{U}:=\left(u, u_{\Gamma}\right)$, $\mathbf{U}_T:=\left(u_T, u_{T, \Gamma}\right)$, and the linear operator $\mathbf{A}: D(\mathbf{A}) \subset \mathbb{L}^2 \rightarrow \mathbb{L}^2$ is given by
$$
\mathbf{A}=\left(\begin{array}{cc}
\Delta & 0 \\
-\partial_\nu & \Delta_{\Gamma}
\end{array}\right), \qquad D(\mathbf{A})=\mathbb{H}^2:=\{(u,u_\Gamma)\in H^2(\Omega) \times H^2(\Gamma) : u_{|\Gamma}=u_\Gamma\}.
$$
In \cite{MMS'17}, it has been proven that $\mathbf{A}$ is self-adjoint and dissipative. Since $\mathbb{H}^{2} \hookrightarrow H^{2}(\Omega) \times H^{2}(\Gamma)$ is continuous, then $\mathbb{H}^{2} \hookrightarrow \mathbb{L}^2$ is compact. By \cite[Proposition 4.25]{EN'00}, we deduce that $\mathbf{A}$ has compact resolvent. Therefore, the operator $\mathbf{A}$ satisfies Assumption I.
\begin{corollary}\label{dycor}
Let $\mathbf{U}_0\in \mathcal{I}_R$ and let $\mathbf{U}(t)$ be the corresponding solution to \eqref{dy1}. Then there exists a constant $\theta \in (0,\frac{1}{2})$ such that
\begin{equation*}
    \|\mathbf{U}(0)\|_{\mathbb{L}^2} \le K R^{1-\theta}\sqrt{1 + \frac{2 \,  T^\alpha}{\alpha\Gamma(\alpha)}} \|\mathbf{U}(T)\|_{\mathbb{L}^2}^\theta.
\end{equation*}
\end{corollary}

\section*{Acknowledgment}
M.\! Yamamoto is supported by Grant-in-Aid for Scientific Research (A) 
20H00117 and Grant-in-Aid for Challenging Research (Pioneering) 21K18142, JSPS.

\end{document}